\documentclass{tac}
\usepackage{amssymb, amsmath, paralist, microtype, stmaryrd, url}
\usepackage[nobreak]{cite}
\usepackage[latin1]{inputenc}
\usepackage[dvips, arrow, matrix, tips, curve]{xy}
\usepackage[british]{babel}
\usepackage{color}
\definecolor{darkgreen}{rgb}{0,0.45,0}
\usepackage[pagebackref,colorlinks,citecolor=darkgreen,linkcolor=darkgreen]{hyperref}
\SelectTips{cm}{10}

\newcommand{\cat}[1]{\mathbf{#1}}

\newcommand{\op}{\mathrm{op}}

\newcommand{\thg}{{\mathord{\text{--}}}}

\newcommand{\cd}[2][]{\vcenter{\hbox{\xymatrix#1{#2}}}}

\newcommand{\C}{{\mathcal C}}

\newcommand{\E}{{\mathcal E}}

\newcommand{\xtor}[1]{\cdl[@1]{{} \ar[r]|-{\object@{|}}^{#1} & {}}}
\newcommand{\tor}{\ensuremath{\relbar\joinrel\mapstochar\joinrel\rightarrow}}

\makeatletter

\def\hookleftarrowfill@{\arrowfill@\leftarrow\relbar{\relbar\joinrel\rhook}}
\def\twoheadleftarrowfill@{\arrowfill@\twoheadleftarrow\relbar\relbar}
\def\leftbararrowfill@{\arrowdoublefill@{\leftarrow\mkern-5mu}\relbar\mapstochar\relbar\relbar}
\def\Leftbararrowfill@{\arrowdoublefill@{\Leftarrow\mkern-2mu}\Relbar\Mapstochar\Relbar\Relbar}
\def\leftringarrowfill@{\arrowdoublefill@{\leftarrow\mkern-3mu}\relbar{\mkern-3mu\circ\mkern-2mu}\relbar\relbar}
\def\lefttriarrowfill@{\arrowfill@{\mathrel\triangleleft\mkern0.5mu\joinrel\relbar}\relbar\relbar}
\def\Lefttriarrowfill@{\arrowfill@{\mathrel\triangleleft\mkern1mu\joinrel\Relbar}\Relbar\Relbar}

\def\hookrightarrowfill@{\arrowfill@{\lhook\joinrel\relbar}\relbar\rightarrow}
\def\hookrightarrowfill@{\arrowfill@{\lhook\joinrel\relbar}\relbar\rightarrow}
\def\twoheadrightarrowfill@{\arrowfill@\relbar\relbar\twoheadrightarrow}
\def\rightbararrowfill@{\arrowdoublefill@{\relbar\mkern-0.5mu}\relbar\mapstochar\relbar\rightarrow}
\def\Rightbararrowfill@{\arrowdoublefill@{\Relbar\mkern-2mu}\Relbar\Mapstochar\Relbar\Rightarrow}
\def\rightringarrowfill@{\arrowdoublefill@\relbar\relbar{\mkern-2mu\circ\mkern-3mu}\relbar{\mkern-3mu\rightarrow}}
\def\righttriarrowfill@{\arrowfill@\relbar\relbar{\relbar\joinrel\mkern0.5mu\mathrel\triangleright}}
\def\Righttriarrowfill@{\arrowfill@\Relbar\Relbar{\Relbar\joinrel\mkern1mu\mathrel\triangleright}}

\def\leftrightarrowfill@{\arrowfill@\leftarrow\relbar\rightarrow}
\def\mapstofill@{\arrowfill@{\mapstochar\relbar}\relbar\rightarrow}

\renewcommand*\xleftarrow[2][]{\ext@arrow 20{20}0\leftarrowfill@{#1}{#2}}
\providecommand*\xLeftarrow[2][]{\ext@arrow 60{22}0{\Leftarrowfill@}{#1}{#2}}
\providecommand*\xhookleftarrow[2][]{\ext@arrow 10{20}0\hookleftarrowfill@{#1}{#2}}
\providecommand*\xtwoheadleftarrow[2][]{\ext@arrow 60{20}0\twoheadleftarrowfill@{#1}{#2}}
\providecommand*\xleftbararrow[2][]{\ext@arrow 10{22}0\leftbararrowfill@{#1}{#2}}
\providecommand*\xLeftbararrow[2][]{\ext@arrow 50{24}0\Leftbararrowfill@{#1}{#2}}
\providecommand*\xleftringarrow[2][]{\ext@arrow 10{26}0\leftringarrowfill@{#1}{#2}}
\providecommand*\xlefttriarrow[2][]{\ext@arrow 80{24}0\lefttriarrowfill@{#1}{#2}}
\providecommand*\xLefttriarrow[2][]{\ext@arrow 80{24}0\Lefttriarrowfill@{#1}{#2}}

\renewcommand*\xrightarrow[2][]{\ext@arrow 01{20}0\rightarrowfill@{#1}{#2}}
\providecommand*\xRightarrow[2][]{\ext@arrow 04{22}0{\Rightarrowfill@}{#1}{#2}}
\providecommand*\xhookrightarrow[2][]{\ext@arrow 00{20}0\hookrightarrowfill@{#1}{#2}}
\providecommand*\xtwoheadrightarrow[2][]{\ext@arrow 03{20}0\twoheadrightarrowfill@{#1}{#2}}
\providecommand*\xrightbararrow[2][]{\ext@arrow 01{22}0\rightbararrowfill@{#1}{#2}}
\providecommand*\xRightbararrow[2][]{\ext@arrow 04{24}0\Rightbararrowfill@{#1}{#2}}
\providecommand*\xrightringarrow[2][]{\ext@arrow 01{26}0\rightringarrowfill@{#1}{#2}}
\providecommand*\xrighttriarrow[2][]{\ext@arrow 07{24}0\righttriarrowfill@{#1}{#2}}
\providecommand*\xRighttriarrow[2][]{\ext@arrow 07{24}0\Righttriarrowfill@{#1}{#2}}

\providecommand*\xmapsto[2][]{\ext@arrow 01{20}0\mapstofill@{#1}{#2}}
\providecommand*\xleftrightarrow[2][]{\ext@arrow 10{22}0\leftrightarrowfill@{#1}{#2}}
\providecommand*\xLeftrightarrow[2][]{\ext@arrow 10{27}0{\Leftrightarrowfill@}{#1}{#2}}

\makeatother

\newcommand{\twocong}[2][0.5]{\ar@{}[#2] \save ?(#1)*{\cong}\restore}
\newcommand{\twoeq}[2][0.5]{\ar@{}[#2] \save ?(#1)*{=}\restore}
\newcommand{\rtwocell}[3][0.5]{\ar@{}[#2] \ar@{=>}?(#1)+/l 0.2cm/;?(#1)+/r 0.2cm/^{#3}}
\newcommand{\ltwocell}[3][0.5]{\ar@{}[#2] \ar@{=>}?(#1)+/r 0.2cm/;?(#1)+/l 0.2cm/^{#3}}
\newcommand{\ltwocello}[3][0.5]{\ar@{}[#2] \ar@{=>}?(#1)+/r 0.2cm/;?(#1)+/l 0.2cm/_{#3}}
\newcommand{\dtwocell}[3][0.5]{\ar@{}[#2] \ar@{=>}?(#1)+/u  0.2cm/;?(#1)+/d 0.2cm/^{#3}}
\newcommand{\dltwocell}[3][0.5]{\ar@{}[#2] \ar@{=>}?(#1)+/ur  0.2cm/;?(#1)+/dl 0.2cm/^{#3}}
\newcommand{\drtwocell}[3][0.5]{\ar@{}[#2] \ar@{=>}?(#1)+/ul  0.2cm/;?(#1)+/dr 0.2cm/^{#3}}
\newcommand{\dthreecell}[3][0.5]{\ar@{}[#2] \ar@3{->}?(#1)+/u  0.2cm/;?(#1)+/d 0.2cm/^{#3}}
\newcommand{\utwocell}[3][0.5]{\ar@{}[#2] \ar@{=>}?(#1)+/d 0.2cm/;?(#1)+/u 0.2cm/_{#3}}
\newcommand{\dtwocelltarg}[3][0.5]{\ar@{}#2 \ar@{=>}?(#1)+/u  0.2cm/;?(#1)+/d 0.2cm/^{#3}}
\newcommand{\utwocelltarg}[3][0.5]{\ar@{}#2 \ar@{=>}?(#1)+/d  0.2cm/;?(#1)+/u 0.2cm/_{#3}}

\newdir{(}{{}*!<0em,-.14em>-\cir<.14em>{l^r}}
\newdir{ (}{{}*!/-5pt/\dir{(}}
\newdir{ >}{{}*!/-5pt/\dir{>}}

\newtheorem{Thm}{Theorem}
\newtheorem{Prop}{Proposition}

\begin{document}
 \leftmargini=2em
\title{Remarks on exactness notions\\pertaining to pushouts}
\author{Richard Garner}
\address{Department of Computing, Macquarie University, NSW 2109, Australia}
\eaddress{richard.garner@mq.edu.au} 
\keywords{Exactness, pushouts, Mal'cev relation, difunctional relation}
\amsclass{18A30, 18B25}
\copyrightyear{2011}
\date{\today}
 \maketitle
\begin{abstract}
We call a finitely complete category \emph{diexact} if every Mal'cev relation admits a pushout which is stable under pullback and itself a pullback. We prove three results relating to diexact categories: firstly, that a category is a pretopos if and only if it is diexact with a strict initial object; secondly, that a category is diexact if and only if it is Barr-exact, and every pair of monomorphisms admits a pushout which is stable and a pullback; and thirdly, that a small category with finite limits and pushouts of Mal'cev spans is diexact if and only if it admits a full structure-preserving embedding into a Grothendieck topos.
\end{abstract}

\newcommand{\pair}[2]{(#1,#2)}
\section{Introduction}
Amongst the first facts that a category theorist will learn about limits and colimits is that certain limit or colimit types suffice for the construction of other ones. Thus, for example, all small colimits may be constructed from small coproducts and coequalisers, or instead, from finite colimits and filtered ones; whilst finite colimits  may in turn be constructed from finite coproducts and coequalisers, or alternatively, from the initial object and pushouts.


Somewhat later, a category theorist becomes cognisant of notions such as regularity, Barr-exactness or extensivity, which involve the existence of finite limits, of certain colimits, and of ``exactness conditions'' expressing the good behaviour of the colimits with respect to the finite limits. In~\cite{Garner2011Lex-colimits}, Lack and the author described how such structures may be recognised as instances of a general theory of cocompleteness, fully the equal of the classical theory, but now existing ``in the lex world''; more precisely, in the $2$-category $\cat{LEX}$ of finitely complete categories and finite-limit preserving functors. 

This paves the way for our studying relative constructibility of exactness notions just as is done for ordinary colimits. The maximal exactness notion is that of being an \emph{infinitary pretopos}, and again this may be constructed from lesser notions in various ways. A category is an infinitary pretopos just when it is infinitary extensive and Barr-exact---thus, having well-behaved coproducts, and well-behaved coequalisers of equivalence relations; alternatively, just when it is a pretopos and admits filtered colimits commuting with finite limits. Now a category is a pretopos just when it is (finitary) extensive and Barr-exact, which is analogous to the construction of finite colimits from finite coproducts and coequalisers; but it is notable that there is no corresponding analogue for the constructibility of finite colimits from pushouts and an initial object. 

The purpose of these remarks is to provide such an analogue. We call a finitely complete category \emph{diexact} if it admits pushouts of Mal'cev (also called difunctional) relations, and every such pushout is stable under pullback and itself a pullback square. We then prove three results relating to diexact categories. The first is the relative constructibility result alluded to above: it says that a category is a pretopos if and only if it is diexact with a strict initial object. The second considers diexactness in the absence of a strict initial object, and shows that this notion is in turn constructible from lesser ones: we prove that a category is diexact just when it is Barr-exact and every pair of monomorphisms admits a pushout which is stable and a pullback; equally,  just when it is Barr-exact and adhesive in the sense of~\cite{Lack2005Adhesive}. Our third result  states that a small category with finite limits and pushouts of Mal'cev spans is diexact just when it admits a structure-preserving full embedding into a Grothendieck topos. It follows that diexactness is an exactness notion in the precise sense delineated in~\cite{Garner2011Lex-colimits}.

Diexactness involves only \emph{connected} colimits and finite limits, and so is stable under passage to the coslice; so that, for example, the category of pointed sets is diexact, though it is not a pretopos as its initial object is a zero object. The property of having filtered colimits commuting with finite limits is also stable under coslicing, and so also possessed by the category of pointed sets; in fact, we will show in future work that the ``$\cat{Set}_\ast$-enriched Grothendieck toposes''---that is, the $\cat{Set}_\ast$-categories arising as localisations of presheaf $\cat{Set}_\ast$-categories---are precisely the locally presentable categories which are diexact, with a zero object, and with filtered colimits commuting with finite limits.

\textbf{Acknowledgements.} Thanks to Robin Cockett for setting in motion the train of thought pursued in this paper by asking the question---here answered in the affirmative---``are pretoposes adhesive?''. Thanks also to members of the Australian Category Seminar for useful comments and suggestions. This work was supported by an Australian Research Council Discovery Project, grant number DP110102360.

\section{The results}
We assume that the reader is familiar with the notions of \emph{regular}, \emph{Barr-exact}, \emph{extensive} and \emph{coherent} category, and of \emph{strict} initial object; they could, for example, consult~\cite[Section~A1]{Johnstone2002Sketches}. 
Recall that a \emph{pretopos} is a finitely complete category which is both extensive and Barr-exact (and thus also coherent). Although a pretopos admits finite coproducts and coequalisers of equivalence relations, it need not admit all coequalisers; consequently, the general pretopos cannot admit all pushouts, since these, together with the initial object, would imply the existence of all finite colimits. However, a pretopos certainly has \emph{some} pushouts; for example, pushouts over the initial object yielding coproducts. What we aim to describe is a class of well-behaved pushouts which exist in any pretopos, and which, in the presence of a strict initial object, completely characterise the pretoposes.

%

By a \emph{Mal'cev span} in a finitely complete category, we mean a jointly monic span $f \colon A \leftarrow C \rightarrow B \colon g$ for which there is a factorisation
\begin{equation*}
\cd[@-1em]{
 & C \times_B C \times_A C \ar@{-->}[d] \ar[ddl]_{f.\pi_1} \ar[ddr]^{g.\pi_3} \\
 & C \ar[dl]^f \ar[dr]_g \\
 A & & B\rlap{ .}}
\end{equation*}
In the category of sets, a relation $R \subset A \times B$ is Mal'cev just when it satisfies the condition $(a \mathrel{R} b) \wedge (a \mathrel R b') \wedge (a' \mathrel R b) \Rightarrow (a' \mathrel R b')$; a span $h \colon A \leftarrow C \rightarrow B$ in the general $\C$ is Mal'cev just when $\C(X,C)$ is a Mal'cev relation in $\cat{Set}$ for each $X \in \C$. 
Some important classes of Mal'cev spans are:\vskip0.5\baselineskip

\begin{compactenum}[(i)]
\item Any span $f \colon A \leftarrow C \rightarrow B \colon g$ in which either $f$ or $g$ is monic;
\item Any span $f \colon A \leftarrow C \rightarrow B \colon g$ which is the pullback of a cospan;
\item Any endospan $s \colon A \leftarrow E \rightarrow A \colon t$ constituting an equivalence relation on $A$. Indeed, an endospan is an equivalence relation if and only if it is Mal'cev and reflexive.
\end{compactenum}\vskip0.5\baselineskip


We call a finitely complete category $\C$ \emph{diexact} if it admits pushouts of Mal'cev spans, and moreover, every such pushout square is stable under pullback and is itself a pullback. Equivalently, a category is diexact if it admits stable pushouts of pullbacks of pairs of arrows, and every Mal'cev span is a pullback span.

\begin{Prop}\label{prop:pretopmalcevexact}
A pretopos is diexact.
\end{Prop}
In proving this, we will exploit the locally preordered \emph{bicategory of relations} $\cat{Rel}(\C)$ associated to a regular category $\C$, as described in~\cite{Carboni1987Cartesian}, for example; its objects are those of $\C$, its morphisms are jointly monic spans in $\C$ and its $2$-cells are span morphisms. For a $\C$ which is coherent, $\cat{Rel}(\C)$ has finite unions in each hom-preorder, preserved by composition on each side. For a $\C$ which is moreover extensive, the finite coproducts of $\C$ extend to $\cat{Rel}(\C)$, there becoming finite biproducts; this allows us to describe relations between coproducts using a matrix calculus, as described in~\cite[Section 6]{Carboni1987Cartesian}. Two further pieces of structure will be important: firstly, the identity-on-objects involution $(\thg)^o \colon \cat{Rel}(\C)^\op \to \cat{Rel}(\C)$ that exchanges domain and codomain; secondly, the identity-on-objects, locally full embedding $\C \to \cat{Rel}(\C)$ which sends a map $f \colon A \to B$ to the jointly monic span $1 \colon A \leftarrow A \rightarrow B \colon f$. We prefer to leave this embedding nameless, using the same name to denote a map in $\C$ and the corresponding morphism of $\cat{Rel}(\C)$. Recall, finally, that for each map $f$ of $\C$, we have $f \dashv f^o$ in $\cat{Rel}(\C)$.

\begin{proof}
Given $f \colon A \leftarrow C \to B \colon g$ a Mal'cev span in $\C$, we let $R = gf^o \colon A \tor B$ in $\cat{Rel}(\C)$; thus $R$ is simply the relation embodied by the given Mal'cev span, with the Mal'cev condition now corresponding to the inequality $RR^oR \leqslant R$.
Consider the relation
\begin{equation*}
E = \left(\begin{matrix}1_A \cup R^oR & R^o \\ R & 1_B \cup RR^o\end{matrix}\right) \colon A+B \tor A+B\rlap{ .}
\end{equation*}
This is clearly reflexive and symmetric, whilst transitivity $EE \leqslant E$ follows by multiplying matrices and using $RR^oR \leqslant R$. So $E$ is an equivalence relation on $A+B$, which, since $\C$ is Barr-exact, admits a coequaliser $[h,k] \colon A + B \to D$. The universal property of this coequaliser, expressed in terms of $\cat{Rel}(\C)$, says that  $h \colon A \to D$ and $k \colon B \to D$ are initial amongst maps with
\begin{equation*}
\left(\begin{matrix}h & k\end{matrix}\right)\left(\begin{matrix}1_A \cup R^oR & R^o \\ R & 1_B \cup RR^o\end{matrix}\right) \leqslant \left(\begin{matrix}h & k\end{matrix}\right) \colon A + B \tor D\rlap{ .}
\end{equation*}
By expanding this condition out, it is easy to see that it is equivalent to the single condition that $kR \leqslant h$; in other words, that $kgf^o \leqslant h$; in other words, that $kg \leqslant hf$ in $\cat{Rel}(\C)$, or in other words, that $kg = hf$ in $\C$.
Thus $h$ and $k$ exhibit $D$ as a pushout of $f$ against $g$, and in fact as a stable pushout, since the construction used only colimits stable under pullback. It remains to show that $f$ and $g$ exhibit $C$ as a pullback of $h$ against $k$. Since $\C$ is Barr-exact, $E$ is the kernel-pair of $[h,k] \colon A+B \to D$, which in terms of $\cat{Rel}(\C)$, says that
\begin{equation*}
E = \left(\begin{matrix}h^o \\ k^o\end{matrix}\right)\left(\begin{matrix}h & k\end{matrix}\right) \colon A + B \tor A + B\rlap{ ;}
\end{equation*}
whence in particular $k^o h = R = gf^o \colon A \tor B$, so that $C$ is the pullback of $h$ against $k$, as required.
%
\end{proof}
%
We now show that well-behaved Mal'cev pushouts, together with a strict initial object, serve to completely characterise pretoposes. This gives us the promised analogue ``in the lex world'' of the construction of all finite colimits from pushouts and the initial object.

\begin{Thm}\label{thm:1}
A finitely complete category $\C$ is a pretopos if and only if it is diexact with a strict initial object.
\end{Thm}
\begin{proof}
If $\C$ is a pretopos, then it certainly has a strict initial object, and we have just seen that it is also is diexact. Suppose conversely that $\C$ is diexact with a strict initial object.
Then for any $A$ and $B$ the unique maps $0 \to A$ and $0 \to B$ are monic; whence $A \leftarrow 0 \rightarrow B$ is a Mal'cev span, and so admits a stable pushout which is also a pullback. Such a pushout is a stable coproduct of $A$ and $B$; that it is also a pullback says that this coproduct is disjoint, so showing that $\C$ is extensive. 
Now any equivalence relation $(s,t) \colon R \rightrightarrows A$ in $\C$ defines a Mal'cev span from $A$ to $A$, which consequently admits a stable pushout which is also a pullback. Since an equivalence relation is a reflexive pair, this pushout is equally a stable coequaliser of $(s,t)$; that it is also a pullback now says that $(s,t)$ is the kernel-pair of its coequaliser, so that $\C$ is Barr-exact. \end{proof}

This result characterises pretoposes in terms of diexactness and a strict initial object. We now examine what happens when we remove the requirement of a strict initial object.
In the following result, we call a finitely complete category $\C$ \emph{adhesive}, as in~\cite{Lack2005Adhesive}, if it admits pushouts along monomorphisms, which are stable and are pullbacks; this is not in fact the definition given in~\cite{Lack2005Adhesive}, but was shown to be equivalent to it in~\cite{GarnerOn-the-axioms}.
We call $\C$ \emph{amalgamable} if every span of monomorphisms  admits a pushout which is stable and a pullback. 

%
%

\begin{Thm}\label{thm:2}
For a finitely complete $\C$, the following are equivalent:\vskip0.3\baselineskip
\begin{compactenum}[(i)]
\item $\C$ is diexact;
\item $\C$ is Barr-exact and adhesive;
\item $\C$ is Barr-exact and amalgamable.
\end{compactenum}
\end{Thm}
The implication (ii) $\Rightarrow$ (i) of this result is essentially stated as~\cite[Theorem 5.2]{Meisen1974Relations}, though under the unnecessary additional assumption that $\C$ have all pushouts. Observe also that since pretoposes are Barr-exact by definition, and are known to be amalgamable---see~\cite[Lemma A1.4.8]{Johnstone2002Sketches} for example---this result gives another proof of our Proposition~\ref{prop:pretopmalcevexact}. 
\begin{proof}
The implications (i) $\Rightarrow$ (ii) $\Rightarrow$ (iii) are trivial, since a span with both legs monomorphic certainly has one leg monomorphic, and a span with one leg monomorphic is Mal'cev; it remains to prove (iii) $\Rightarrow$ (i). Suppose that $\C$ is finitely complete, Barr-exact and amalgamable. It follows that $\C$ admits stable binary unions: for given monomorphisms $A \rightarrowtail C \leftarrowtail B$, we may form their intersection $A\cap B$ and the pushout $P$ of the projections $A \leftarrowtail A\cap B \rightarrowtail B$; now the stability of this pushout ensures that the induced map $P \to C$ is monomorphic---see~\cite[Theorem 5.1]{Lack2005Adhesive}, for example---so that $P$ is the stable union of the subobjects $A$ and $B$. It follows that $\cat{Rel}(\C)$ admits binary unions in each of its hom-preorders, preserved by composition on both sides.

We first prove that any Mal'cev span $f \colon A \leftarrow C \twoheadrightarrow B \colon g$ in $\C$ with one leg regular epimorphic admits a stable pushout which is a pullback. Given such a span, we consider the solid part of the diagram
\begin{equation*}
\cd{
M \ar@<3pt>[r]^{c_0} \ar@<-3pt>[r]_{c_1} \ar[d]_1 & C \ar@{->>}[r]^g \ar[d]^f & B \ar@{-->}[d]^k \\
M \ar@<3pt>[r]^{fc_0} \ar@<-3pt>[r]_{fc_1} & A \ar@{-->>}[r]_h & D\rlap{ ,}
}
\end{equation*}
with $(c_0, c_1)$ the kernel-pair of $g$. Taking $R = gf^o \colon A \tor B$ in $\cat{Rel}(\C)$ as before, the Mal'cev condition $RR^o R \leqslant R$ implies that $1_A \cup R^oR$ is an equivalence relation on $A$. Let $h \colon A \twoheadrightarrow D$ be its coequaliser; so $h$ is initial amongst maps such that $h(1_A \cup R^oR) \leqslant h$, or equally, such that $hR^oR \leqslant h$. But $hR^oR = hfg^ogf^o = hfc_0c_1^of^o$, and so $h$ is initial such that $hfc_0 \leqslant hfc_1$ in $\cat{Rel}(\C)$, or such that $hfc_0 = hfc_1$ in $\C$. In other words, $h$ is a coequaliser of $(fc_0, fc_1)$, and so both rows of the above diagram are coequalisers. We therefore induce a unique map $k \colon B \to D$ as indicated; and since the left-hand vertical map is an identity, it follows that the right square is a pushout. This pushout is clearly stable under pullback, having been constructed from stable colimits; to show that it itself is a pullback is equally to show that $R = k^oh$ in $\cat{Rel}(\C)$. Because $g$ is strong epimorphic, we have $1_B = gg^o$ and so $k^oh = gg^ok^oh = gf^oh^oh = Rh^oh$. But because $\C$ is Barr-exact, $h^o h = 1_A \cup R^oR$ and so $k^oh = R(1_A\cup R^oR) = R \cup RR^oR = R$ as required.

We now prove that $\C$ is diexact. Given a Mal'cev span $f \colon A \leftarrow C \rightarrow B$ in $\C$, form a (regular epi, mono) factorisation $g = g_2 \circ g_1$. It is easy to see that the pair $(f, g_1)$ is again a Mal'cev span, which by the case just proved admits a stable pushout and pullback, as on the left in:
\begin{equation*}
\cd{
 C \ar[d]_f \ar@{->>}[r]^{g_1} & C' \ar@{ >->}[r]^{g_2} \ar[d]^{f'} & B \\
 A \ar@{->>}[r] & A'
} \qquad
\cd{
 C' \ar@{ >->}[d]_{g_2} \ar@{->>}[r]^{f'_1} & C'' \ar@{ >->}[r]^{f'_2} \ar[d]^{g'_2} & A' \rlap{ .} \\
 B \ar@{->>}[r] & B'
} \qquad
\end{equation*}
Now forming a (regular epi, mono) factorisation $f' = f'_2 \circ f'_1$, we obtain a Mal'cev span $(f'_1, g_2)$ which since $f'_1$ is regular epimorphic, admits a stable pushout and pullback, as on the right above. On pulling back this pushout square along $g'_2$, its left edge becomes invertible; since the resultant square is still a pushout, its right edge must also be invertible, which is to say that $g'_2$ has trivial kernel-pair, and so is monomorphic. So $(f'_2, g'_2)$ is a pair of monomorphisms, which, as $\C$ is amalgamable, admit a stable pushout and pullback. Pasting together the three pushout squares just obtained, we conclude that $(f,g)$ admits a stable pushout and pullback as required.
%
%
%
%
\end{proof}

For our third and final result, we prove an embedding theorem for diexact categories, showing that they capture precisely the compatibilities holding between finite limits and pushouts of Mal'cev spans that hold in any Grothendieck topos. It follows from this that diexactness is an exactness notion in the sense of~\cite{Garner2011Lex-colimits}.
\begin{Thm}
For $\C$ a small, finitely complete category admitting pushouts of Mal'cev spans, the following are equivalent:\vskip0.5\baselineskip
\begin{compactenum}[(i)]
\item $\C$ is diexact;
\item $\C$ admits a full embedding into a Grothendieck topos via a functor preserving finite limits and pushouts of Mal'cev spans.
\end{compactenum}
\end{Thm}
\begin{proof}
Any Grothendieck topos is a pretopos, hence diexact; and clearly any full subcategory of a diexact category closed under the relevant limits and colimits will again be diexact. So if $J \colon \C \to \E$ is an embedding of $\C$ into a Grothendieck topos, then the essential image of $J$ is diexact; but $\C$ is equivalent to this essential image, and hence is itself diexact. This shows that (ii) $\Rightarrow$ (i).

Conversely, suppose that $\C$ is diexact. We consider the smallest topology on $\C$ for which every pair of maps $A \to D \leftarrow B$ arising as the pushout of some Mal'cev span is covering. Since such covers are stably effective-epimorphic, they generate a subcanonical topology on $\C$ and now the restricted Yoneda embedding provides a limit-preserving full embedding $J \colon \C \to \cat{Sh}(\C)$. We must show that $J$ also preserves pushouts of Mal'cev spans; for this, it suffices to show that it preserves coequalisers of equivalence relations and pushouts of pairs of monomorphisms, since Theorem~\ref{thm:2} used only these colimits (together with finite limits) to construct pushouts of Mal'cev spans in a diexact category.

If $(s,t) \colon E \rightrightarrows A$ is an equivalence relation in $\C$, then its coequaliser $q \colon A \twoheadrightarrow B$ is a singleton cover, and so $Jq$ is an epimorphism in a topos, hence the coequaliser of its kernel-pair. But that kernel-pair is the image under $J$ of $q$'s kernel-pair $(s,t)$, whence $J$ preserves coequalisers of equivalence relations. On the other hand, if $f \colon A \leftarrowtail B \rightarrowtail C \colon g$ is a pair of monomorphisms in $\C$, then they admit a pushout $h \colon A \rightarrowtail D \leftarrowtail C \colon k$ where, arguing as in the proof of Theorem~\ref{thm:2}, both $h$ and $k$ are monomorphic. Since $(h,k)$ comprise a covering family, $(Jh, Jk)$ are a jointly epimorphic pair of monomorphisms in a topos, and are thus the pushout of their own pullback; but that pullback is the image under $J$ of the pullback of $h$ and $k$, which is $(f,g)$. Thus $J$ preserves pushouts of pairs of monomorphisms as well as coequalisers of equivalence relations; and so as argued above, it also preserves pushouts of Mal'cev spans.\end{proof}


\bibliographystyle{acm}

\bibliography{bibdata}

\end{document}